\documentclass[11pt]{article}

\usepackage[a4paper,margin=1.15in]{geometry}
\usepackage{amsmath,amssymb,amsthm,mathtools}
\usepackage{mathrsfs}
\usepackage{enumitem}
\usepackage[colorlinks=true,linkcolor=blue,citecolor=blue,urlcolor=blue]{hyperref}
\numberwithin{equation}{section}

\newtheorem{theorem}{Theorem}[section]
\newtheorem{proposition}[theorem]{Proposition}
\newtheorem{lemma}[theorem]{Lemma}
\newtheorem{corollary}[theorem]{Corollary}
\newtheorem{definition}[theorem]{Definition}

\newcommand{\R}{\mathbb R}
\newcommand{\Sn}{\mathbb S^{n-1}}
\newcommand{\bd}{\partial}
\newcommand{\Om}{\Omega}

\newcommand{\DN}{\Lambda}
\newcommand{\Xray}{\mathcal X}

\newcommand{\Id}{\mathrm{Id}}
\newcommand{\tr}{\operatorname{tr}}

\title{An inverse source problem for the Monge--Ampere equation from large boundary data}
\author{%
C\u{a}t\u{a}lin I. C\^{a}rstea\thanks{
Department of Applied Mathematics,
National Yang Ming Chiao Tung University,
Hsinchu 300, Taiwan, R.O.C.,
\texttt{catalin.carstea@gmail.com}}
\and
Tuhin Ghosh\thanks{
Harish-Chandra Research Institute, Homi Bhabha National Institute,
Chhatnag Road, Jhunsi, Prayagraj (Allahabad) 211 019, India,
\texttt{tuhinghosh@hri.res.in}}}
\date{}

\begin{document}

\maketitle

\begin{abstract}
We study an inverse source problem for the Monge--Ampere equation
\[
    \det D^2u=f(x)
\]
on a bounded smooth uniformly convex domain.  In the smooth classical regime, we prove that the Dirichlet-to-Neumann map associated with convex solutions determines the positive source uniquely.  The proof uses a family of large boundary values and reduces the inverse source problem to the injectivity of the Euclidean X-ray transform.
\end{abstract}

\section{Introduction}

Let \(\Om\subset\R^n\), \(n\ge2\), be a bounded smooth uniformly convex domain.  This paper concerns the inverse source problem for the Monge--Ampere equation
\begin{equation}\label{eq:MA}
    \det D^2u=f(x)\quad\text{in }\Om,
\end{equation}
where the unknown source \(f\) is positive.  We ask whether boundary measurements for the corresponding convex solutions determine \(f\).

We use the standard classical Dirichlet theory for the Monge--Ampere equation.  If
\[
    f\in C^\infty(\overline\Om),
    \qquad
    0<c_f\le f\le C_f<\infty,
\]
and \(g\in C^\infty(\partial\Om)\), then the Dirichlet problem
\begin{equation}\label{eq:MA-intro-dirichlet}
    \det D^2u=f
    \quad\text{in }\Om,
    \qquad
    u|_{\partial\Om}=g,
\end{equation}
has a unique smooth strictly convex solution \(u_g^f\).  The associated Dirichlet-to-Neumann map is
\begin{equation}\label{eq:DN-intro}
    \DN_f(g)=\partial_\nu u_g^f|_{\partial\Om},
\end{equation}
where \(\nu\) is the outward unit normal.  The forward solvability statement follows from the classical work of Caffarelli--Nirenberg--Spruck \cite{CaffarelliNirenbergSpruck1984}; see also Guti\'errez \cite{Gutierrez2016} for background on the Monge--Ampere equation.

Our main result is the following uniqueness theorem.

\begin{theorem}\label{thm:main}
Let \(\Om\subset\R^n\), \(n\ge2\), be bounded, smooth, and uniformly convex.  Let
\[
    f_1,f_2\in C^\infty(\overline\Om),
    \qquad
    0<c\le f_j\le C<\infty.
\]
If
\[
    \DN_{f_1}(g)=\DN_{f_2}(g)
    \qquad\text{for every }g\in C^\infty(\partial\Om),
\]
then
\[
    f_1=f_2\quad\text{in }\Om.
\]
\end{theorem}

Inverse boundary value problems for nonlinear equations have developed in several related directions.  For semilinear elliptic and parabolic equations, early uniqueness results include \cite{Isakov1993,IsakovSylvester1994,IsakovNachman1995,Sun2010}.  Later works treat partial data, geometric settings, and power-type or gradient nonlinearities; see, for instance, \cite{FeizmohammadiOksanen2020,KrupchykUhlmann2020,LassasLiimatainenLinSalo2020,LassasLiimatainenLinSalo2021}.  A common method is to prescribe boundary values depending on small parameters and differentiate the nonlinear Dirichlet-to-Neumann map with respect to those parameters.  The resulting identities are linear inverse problems for the coefficients appearing in the nonlinear terms.  In elliptic problems this is the higher order linearization method.  Related ideas also play an important role in hyperbolic inverse problems; see, for example, \cite{KurylevLassasUhlmann2018}.

There is a parallel literature for quasilinear elliptic equations and nonlinear conductivity-type problems.  Early uniqueness results include \cite{Sun1996,SunUhlmann1997,HervasSun2002,KangNakamura2002}.  More recent results and tools include \cite{EggerPietschmannSchlottbom2014,MunozUhlmann2020,Shankar2020,CarsteaFeizmohammadi2021,CarsteaFeizmohammadi2023,KianKrupchykUhlmann2023}.  Related nonlinear geometric inverse problems include inverse problems for minimal surface type equations \cite{CarsteaLassasLiimatainenOksanen2024}.

Inverse problems for the \(p\)-Laplacian and weighted \(p\)-Laplacian form another important part of this development.  These problems are technically different from uniformly elliptic semilinear models because the linearized operators may become degenerate or singular at critical points of the background solution.  Results in this direction include boundary determination, monotonicity and enclosure methods, size estimates, and coefficient recovery; see, for instance, \cite{SaloZhong2012,BranderKarSalo2015,Brander2016,GuoKarSalo2016,BranderHarrachKarSalo2018} and \cite{KarWang2017,CarsteaKar2020,CarsteaFeizmohammadi2025}.  A related nonstandard-growth model is treated in \cite{CarsteaZimmermann2025}.  Parabolic inverse problems for degenerate nonlinear diffusion provide a further related class of nonlinear models; see, for example, \cite{CarsteaGhoshUhlmann2023,CarsteaGhoshNakamura2025,CarsteaGhosh2026}.

The closest previous work to ours is \cite{LiimatainenLin2025}.  It proves uniqueness for the same inverse source problem in two-dimensional convex Euclidean domains.  Their method starts from the full nonlinear Dirichlet-to-Neumann map and linearizes it at non-zero solutions.  If \(u\) solves \(\det D^2u=f\), then the linearized equation has coefficient matrix \(\operatorname{cof}(D^2u)\).  Since \(\operatorname{cof}(D^2u)=f(D^2u)^{-1}\), this equation has the same solutions as the non-divergence form equation \((D^2u)^{-1}_{ij}\partial_{ij}v=0\).  They then  recover the Hessian matrix \(D^2u\) from boundary measurements for this linear equation, and taking the determinant gives $f$.  

The  argument in the present paper is different. We use a large-parameter family of boundary values for which the leading term is a degenerate cylindrical profile.  The first correction satisfies a one-dimensional equation on chords of \(\Omega\), and the boundary normal derivative of this correction gives the corresponding chord integrals of the source.  Thus the inverse source problem is reduced to injectivity of the Euclidean X-ray transform. The method works in all dimensions \(n\ge2\).

Although Theorem~\ref{thm:main} is stated for the full Dirichlet-to-Neumann map, the proof uses only a restricted family of boundary values.  For \(\omega\in S^{n-1}\), let
\[
    \phi_\omega(x)=\frac12|P_{\omega^\perp}x|^2,
\]
where \(P_{\omega^\perp}\) is orthogonal projection onto \(\omega^\perp\).  We prescribe the large cylindrical boundary values
\[
    g_{t,\omega}=t\phi_\omega|_{\partial\Om},
    \qquad t\gg1.
\]
The leading profile \(t\phi_\omega\) has strong convexity in directions transverse to \(\omega\), but no convexity in the \(\omega\)-direction.  The missing longitudinal second derivative appears in the first correction.  Along each chord of \(\Omega\) parallel to \(\omega\), the correction solves a one-dimensional equation whose endpoint derivatives determine the line integral of \(f\) over that chord.  Equality of the Dirichlet-to-Neumann maps on the large cylindrical family therefore gives equality of the X-ray transforms of the two sources.  The standard Fourier-slice injectivity theorem then gives \(f_1=f_2\); see \cite{Natterer2001} for the tomographic background.

The paper is organized as follows.  Section~\ref{sec:forward} records the forward problem and the Dirichlet-to-Neumann map.  Section~\ref{sec:large-cylindrical-data} introduces the large cylindrical boundary values, the associated chord geometry, and the one-dimensional correction that appears in the asymptotic expansion.  Section~\ref{sec:large-data-asymptotics} proves the large-data asymptotic theorem.  Section~\ref{sec:recovery-line-integrals} converts the boundary asymptotic into X-ray data and recalls the injectivity needed for the final step.  Section~\ref{sec:uniqueness-proofs} proves the restricted large-data uniqueness theorem and then the main theorem.

\section{Forward problem and boundary measurements}\label{sec:forward}

We use two standard facts about convex solutions of the Monge--Ampere equation.  The first is the classical smooth Dirichlet theorem.

\begin{theorem}[{\cite[Theorem~1.1]{CaffarelliNirenbergSpruck1984}}]\label{thm:CNS-classical-dirichlet}
Let \(D\subset\R^n\) be a bounded domain with smooth strictly convex boundary.  If \(\psi\in C^\infty(\overline D)\) is positive and \(\varphi\in C^\infty(\partial D)\), then
\[
    \det D^2u=\psi
    \quad\text{in }D,
    \qquad
    u=\varphi
    \quad\text{on }\partial D
\]
has a unique strictly convex solution \(u\in C^\infty(\overline D)\). 
\end{theorem}

The second is the comparison principle for the Monge--Ampere measure.  If \(u\) is convex, we write \(M_u\) for its Monge--Ampere measure; when \(u\in C^2\), this measure is \(M_u=(\det D^2u)\,dx\).

\begin{theorem}[{\cite[Theorem~1.4.6]{Gutierrez2016}}]\label{thm:MA-comparison}
Let \(D\subset\R^n\) be open and bounded, and let \(u,v\in C(\overline D)\) be convex in \(D\).  If \(M_u\le M_v\) in \(D\), then
\[
    \min_{\overline D}(u-v)=\min_{\partial D}(u-v).
\]
In particular, if \(u\ge v\) on \(\partial D\), then \(u\ge v\) in \(D\).
\end{theorem}

We use the following convention throughout the paper.  The domain \(\Om\subset\R^n\) is bounded, smooth, and uniformly convex.  The outward unit normal is denoted by \(\nu\).  Sources belong to the class
\begin{equation}\label{eq:source-class}
    \mathcal F
    =
    \left\{
        f\in C^\infty(\overline\Om):
        0<c_f\le f\le C_f<\infty
    \right\}.
\end{equation}
For \(f\in\mathcal F\), Theorem~\ref{thm:CNS-classical-dirichlet} gives, for every \(g\in C^\infty(\bd\Om)\), a unique smooth strictly convex solution \(u_g^f\in C^\infty(\overline\Om)\) of
\begin{equation}\label{eq:admissible-dirichlet-problem}
    \det D^2u=f
    \quad\text{in }\Om,
    \qquad
    u=g
    \quad\text{on }\bd\Om .
\end{equation}
The Dirichlet-to-Neumann map is
\begin{equation}\label{eq:DN-definition}
    \DN_f:g\mapsto \partial_\nu u_g^f|_{\bd\Om}.
\end{equation}

The paper is written in this smooth framework to keep the inverse problem separate from lower-regularity trace issues.  A lower-regularity formulation would require replacing \eqref{eq:DN-definition} by a boundary trace statement justified by additional regularity.

\section{Large cylindrical data and the chordwise correction}\label{sec:large-cylindrical-data}

We now introduce the special boundary values used in the proof of uniqueness.  The point of this section is to explain why large cylindrical data should produce chord integrals of the source.  The rigorous asymptotic estimate is stated at the end of the section and proved in the next one.

\subsection{Large cylindrical boundary values}

For \(\omega\in\Sn\), let
\[
    P_{\omega^\perp}=\Id-\omega\otimes\omega.
\]
The cylindrical profile is
\begin{equation}\label{eq:phi-omega}
    \phi_\omega(x)=\frac12|P_{\omega^\perp}x|^2.
\end{equation}
The large cylindrical trace family is
\begin{equation}\label{eq:cylindrical-trace-family}
    \mathcal C_{T}
    =
    \{g_{t,\omega}=t\phi_\omega|_{\bd\Om}:t\ge T,\ \omega\in\Sn\}.
\end{equation}
For \(f\in\mathcal F\), we write \(u_{t,\omega}^f\) for the solution of
\begin{equation}\label{eq:large-data-equation-intro}
    \det D^2u_{t,\omega}^f=f
    \quad\text{in }\Om,
    \qquad
    u_{t,\omega}^f=g_{t,\omega}
    \quad\text{on }\bd\Om.
\end{equation}
The corresponding measurements are
\begin{equation}\label{eq:DN-large-family}
    \DN_f(g_{t,\omega})
    =
    \partial_\nu u_{t,\omega}^f|_{\bd\Om}.
\end{equation}

\subsection{Chord geometry}\label{subsec:chord-geometry}

Fix \(\omega\in\Sn\).  Let
\[
    \Pi_\omega=\omega^\perp,
    \qquad
    \Omega_\omega=P_{\omega^\perp}\Omega\subset \Pi_\omega .
\]
For \(y\in\Omega_\omega\), set
\begin{equation}\label{eq:s-plus-minus-def}
\begin{aligned}
    s_+(y)&=\sup\{s\in\R:y+s\omega\in\Omega\},\\
    s_-(y)&=\inf\{s\in\R:y+s\omega\in\Omega\}.
\end{aligned}
\end{equation}
The corresponding endpoints of the chord through $y$ are
\begin{equation}\label{eq:chord-endpoints}
    x_\pm(y)=y+s_\pm(y)\omega\in\bd\Om.
\end{equation}
We also introduce the incoming, outgoing, and glancing boundary sets
\begin{equation}\label{eq:incoming-outgoing-glancing}
\begin{aligned}
    \Gamma_\omega^+&=\{x\in\bd\Om:\omega\cdot\nu(x)>0\},\\
    \Gamma_\omega^-&=\{x\in\bd\Om:\omega\cdot\nu(x)<0\},\\
    \Gamma_\omega^{\mathrm{gl}}&=\{x\in\bd\Om:\omega\cdot\nu(x)=0\}.
\end{aligned}
\end{equation}
Thus
\[
    \Gamma_\omega=\Gamma_\omega^+\cup\Gamma_\omega^-
    =\bd\Om\setminus\Gamma_\omega^{\mathrm{gl}}
\]
is the non-glancing boundary for the direction \(\omega\).

\begin{lemma}\label{lem:chord-parametrization}
Let \(\Om\subset\R^n\) be bounded, smooth, and strictly convex.  Then \(\Omega_\omega\) is an open subset of \(\Pi_\omega\).  For every \(y\in\Omega_\omega\),
\begin{equation}\label{eq:chord-interval}
    \Omega\cap (y+\R\omega)
    =
    \{y+s\omega:s_-(y)<s<s_+(y)\},
\end{equation}
with \(s_-(y)<s_+(y)\).  Moreover \(x_\pm(y)\in\Gamma_\omega^\pm\), that is,
\begin{equation}\label{eq:normal-signs}
    \omega\cdot \nu(x_+(y))>0,
    \qquad
    \omega\cdot \nu(x_-(y))<0.
\end{equation}
The maps
\[
    y\mapsto s_\pm(y),
    \qquad
    y\mapsto x_\pm(y),
\]
are smooth on \(\Omega_\omega\).  Equivalently, the projection map \(P_{\omega^\perp}\) restricts to smooth diffeomorphisms
\[
    \Gamma_\omega^\pm \longrightarrow \Omega_\omega,
\]
with inverses \(y\mapsto x_\pm(y)\).
\end{lemma}

\begin{proof}
Since \(P_{\omega^\perp}:\R^n\to\Pi_\omega\) is an open map, \(\Omega_\omega\) is open.  For fixed \(y\in\Omega_\omega\), the set
\[
    I_y=\{s\in\R:y+s\omega\in\Omega\}
\]
is a nonempty bounded open interval because \(\Omega\) is open, bounded, and convex.  This proves \eqref{eq:chord-interval} and gives \(s_-(y)<s_+(y)\).  Since \(\Omega\) is bounded and open, the endpoints \(x_\pm(y)=y+s_\pm(y)\omega\) lie on \(\bd\Om\).

We next prove non-glancing and the signs.  Let \(\rho\) be a smooth defining function with
\[
    \Omega=\{\rho<0\},
    \qquad
    \bd\Omega=\{\rho=0\},
    \qquad
    \nabla\rho=|\nabla\rho|\nu\quad\text{on }\bd\Omega .
\]
If \(\omega\cdot\nu(x_+(y))=0\), then the line \(x_+(y)+\R\omega\) is contained in the tangent hyperplane to \(\bd\Omega\) at \(x_+(y)\).  By strict convexity, this supporting hyperplane meets \(\overline\Omega\) only at \(x_+(y)\).  This contradicts the fact that \(y+s\omega\in\Omega\) for \(s<s_+(y)\) close to \(s_+(y)\).  Hence \(\omega\cdot\nu(x_+(y))\ne0\).  Similarly \(\omega\cdot\nu(x_-(y))\ne0\).

The sign follows from the crossing direction.  The function \(s\mapsto \rho(y+s\omega)\) is negative on \((s_-(y),s_+(y))\) and zero at the endpoints.  Hence
\[
    \frac{d}{ds}\rho(y+s\omega)\bigg|_{s=s_+(y)}>0,
    \qquad
    \frac{d}{ds}\rho(y+s\omega)\bigg|_{s=s_-(y)}<0.
\]
Since this derivative equals \(|\nabla\rho|\nu\cdot\omega\) on the boundary, \eqref{eq:normal-signs} follows.

Finally, because \(\partial_s\rho(y+s_\pm(y)\omega)\ne0\), the implicit function theorem gives smoothness of \(s_\pm\).  The diffeomorphism statement follows from the identity
\[
    P_{\omega^\perp}(x_\pm(y))=y
\]
and uniqueness of the two endpoints on each chord.
\end{proof}

\subsection{The chordwise correction}\label{subsec:chordwise-correction}

For \(f\in C^\infty(\overline\Om)\), define \(w_\omega^f\) chordwise by
\begin{equation}\label{eq:w-correction}
    \partial_s^2 w_\omega^f(y,s)
    =
    f(y+s\omega),
    \qquad
    w_\omega^f(y,s_\pm(y))=0 .
\end{equation}
It is useful to introduce the chord length
\begin{equation}\label{eq:chord-length}
    \ell_\omega(y)=s_+(y)-s_-(y).
\end{equation}
Solving the one-dimensional Dirichlet problem on each chord gives the explicit formula
\begin{equation}\label{eq:w-explicit}
\begin{aligned}
    w_\omega^f(y,s)
    &=
    \int_{s_-(y)}^s (s-r) f(y+r\omega)\,dr \\
    &\quad
    -
    \frac{s-s_-(y)}{\ell_\omega(y)}
    \int_{s_-(y)}^{s_+(y)}
        (s_+(y)-r) f(y+r\omega)\,dr .
\end{aligned}
\end{equation}
In particular \(w_\omega^f\) is smooth in the open set
\[
    \{(y,s):y\in\Omega_\omega,\ s_-(y)<s<s_+(y)\},
\]
after identifying \((y,s)\) with \(y+s\omega\).  The following estimates record precisely what will be used later.

\begin{lemma}\label{lem:chordwise-C0-bounds}
Let \(f\in C(\overline\Om)\).  For every \(\omega\in\Sn\), every \(y\in\Omega_\omega\), and every \(s\in[s_-(y),s_+(y)]\),
\begin{equation}\label{eq:chordwise-basic-bounds}
    |w_\omega^f(y,s)|
    \le
    \frac18 \|f\|_{L^\infty(\Om)}\ell_\omega(y)^2,
    \qquad
    |\partial_s w_\omega^f(y,s)|
    \le
    \|f\|_{L^\infty(\Om)}\ell_\omega(y),
\end{equation}
and
\begin{equation}\label{eq:chordwise-second-s-bound}
    |\partial_s^2w_\omega^f(y,s)|
    \le
    \|f\|_{L^\infty(\Om)}.
\end{equation}
Consequently \(w_\omega^f=O(\ell_\omega^2)\) and \(\partial_s w_\omega^f=O(\ell_\omega)\) as the chord length tends to zero.
\end{lemma}

\begin{proof}
For fixed \(y\), write \(a=s_-(y)\), \(b=s_+(y)\), and \(\ell=b-a\).  The Dirichlet Green function for \(\partial_s^2\) on \((a,b)\) satisfies
\[
    |G(s,r)|\le \frac{\ell}{4},
    \qquad
    \int_a^b |G(s,r)|\,dr\le \frac{\ell^2}{8}.
\]
Equivalently, these bounds follow directly from \eqref{eq:w-explicit}.  Thus
\[
    |w(y,s)|\le \|f\|_{L^\infty}\frac{\ell^2}{8}.
\]
Differentiating \eqref{eq:w-explicit} with respect to \(s\) gives
\[
    \partial_s w(y,s)
    =
    \int_a^s f(y+r\omega)\,dr
    -
    \frac1\ell
    \int_a^b (b-r)f(y+r\omega)\,dr,
\]
which implies \(|\partial_s w(y,s)|\le \ell\|f\|_{L^\infty}\).  The last bound is just the equation \(\partial_s^2w=f\).
\end{proof}

If a boundary point is glancing for the direction \(\omega\), then it is the limit of endpoints of chords whose lengths tend to zero.  The bound \eqref{eq:chordwise-basic-bounds} implies \(w_\omega^f=O(\ell_\omega^2)\) along such chords.  Thus \(w_\omega^f\) extends continuously to the glancing boundary by setting it equal to zero there.  Whenever we write \(t\phi_\omega+t^{1-n}w_\omega^f=t\phi_\omega\) on all of \(\bd\Om\), this extension is understood.  For smooth \(f\), Lemma~\ref{lem:global-glancing-smoothness} below gives a stronger conclusion: for each fixed direction \(\omega\), the apparent singularities of the endpoint parametrization cancel and \(w_\omega^f\) extends smoothly through glancing.

\subsection{The large-data asymptotic}\label{subsec:large-data-asymptotic-statement}

Fix \(\omega\in S^{n-1}\) and let \(\phi_\omega\) be defined by \eqref{eq:phi-omega}.  For \(t\gg1\), let \(u_{t,\omega}^f\) be the classical convex solution of
\begin{equation}\label{eq:large-data-problem}
    \det D^2u_{t,\omega}^f=f
    \quad\text{in }\Om,
    \qquad
    u_{t,\omega}^f=t\phi_\omega
    \quad\text{on }\bd\Om.
\end{equation}

We make the Ansatz
\begin{equation}\label{eq:large-data-formal}
    u_{t,\omega}^f
    =
    t\phi_\omega+t^{1-n}w_\omega^f+r_{t,\omega}^f.
\end{equation}
Indeed, in coordinates \(x=y+s\omega\),
\[
    D^2(t\phi_\omega+t^{1-n}w)
    =
    \begin{pmatrix}
        tI_{n-1}+t^{1-n}D_y^2w
        &
        t^{1-n}D_y\partial_s w \\[1mm]
        t^{1-n}\partial_sD_yw
        &
        t^{1-n}\partial_s^2w
    \end{pmatrix}.
\]
A determinant expansion gives
\begin{equation}\label{eq:det-expansion}
    \det D^2(t\phi_\omega+t^{1-n}w)
    =
    \partial_s^2w+O(t^{-n})
\end{equation}
provided \(w\) has bounded second derivatives.  Lemma~\ref{lem:non-glancing-regularity} supplies such bounds for \(w_\omega^f\) on regions separated from glancing.  Therefore the leading correction should satisfy \(\partial_s^2w=f\), with zero endpoint values on every chord.

\begin{theorem}\label{thm:large-data-asymptotic}
Let \(\Om\subset\R^n\) be bounded, smooth, and uniformly convex, and let \(f\in C^\infty(\overline\Omega)\) satisfy \(f\ge c_f>0\).  Fix \(\omega\in S^{n-1}\).  Let \(u_{t,\omega}^f\) solve \eqref{eq:large-data-problem}, and let \(w_\omega^f\) be defined by \eqref{eq:w-correction}.  Then there are constants \(C_\omega,T_\omega<\infty\) such that, for all \(t\ge T_\omega\),
\begin{equation}\label{eq:bulk-asymptotic}
    \|u_{t,\omega}^f-(t\phi_\omega+t^{1-n}w_\omega^f)\|_{L^\infty(\Omega)}
    \le
    C_\omega t^{1-2n}.
\end{equation}
Moreover, for every compact set \(K\Subset\Gamma_\omega\), there are constants \(C_K,T_K<\infty\) such that, for all \(t\ge T_K\),
\begin{equation}\label{eq:DN-asymptotic}
    \sup_{x\in K}
    \left|
    \partial_\nu u_{t,\omega}^f(x)
    -
    t\partial_\nu\phi_\omega(x)
    -
    t^{1-n}\partial_\nu w_\omega^f(x)
    \right|
    \le
    C_Kt^{1-n-n/2}.
\end{equation}
In particular, the error in \eqref{eq:DN-asymptotic} is \(o(t^{1-n})\) locally uniformly on \(\Gamma_\omega\).  The constants are not asserted to be uniform as \(\omega\) varies.
\end{theorem}

The proof of Theorem~\ref{thm:large-data-asymptotic} is given in Section~\ref{sec:large-data-asymptotics}.  The bulk estimate comes from Proposition~\ref{prop:comparison-barriers}, and the boundary normal estimate comes from Proposition~\ref{prop:boundary-normal-from-C0}.

\section{Proof of the large-data asymptotic}\label{sec:large-data-asymptotics}

This section proves Theorem~\ref{thm:large-data-asymptotic} for each fixed direction \(\omega\).  The proof first records the regularity of the chordwise correction, then uses an exact Schur-complement determinant expansion.  Global comparison barriers give the \(C^0\) asymptotic, and localized boundary barriers upgrade it to the normal derivative asymptotic on compact subsets of the non-glancing boundary.

\subsection{Regularity of the chordwise correction}\label{subsec:regularity-chordwise-correction}

For derivative estimates one must stay away from glancing chords.  For \(0<\kappa<1\), define
\begin{equation}\label{eq:non-glancing-projected-set}
\begin{aligned}
    \Omega_{\omega,\kappa}
    =
    \{y\in\Omega_\omega:
       \omega\cdot\nu(x_+(y))\ge\kappa,
       \ -\omega\cdot\nu(x_-(y))\ge\kappa\},
\end{aligned}
\end{equation}
and set
\begin{equation}\label{eq:non-glancing-tube}
    Q_{\omega,\kappa}
    =
    \{y+s\omega:y\in\Omega_{\omega,\kappa},\ s_-(y)\le s\le s_+(y)\}.
\end{equation}
We also write
\[
    \Gamma_{\omega,\kappa}^\pm=x_\pm(\Omega_{\omega,\kappa})\subset \Gamma_\omega^\pm.
\]
When \(C^m\)-norms are taken over \(\Omega_{\omega,\kappa}\) or \(Q_{\omega,\kappa}\), they mean suprema of the corresponding coordinate derivatives on these sets.  Equivalently, one may take the norms on a slightly larger open non-glancing neighborhood.

\begin{lemma}\label{lem:non-glancing-regularity}
Let \(m\ge0\) and \(0<\kappa<1\).  There is a constant
\(C=C(\Omega,m,\kappa)\), independent of \(\omega\), such that the following hold.

First,
\begin{equation}\label{eq:s-pm-Cm-estimates}
    \|s_+\|_{C^m(\Omega_{\omega,\kappa})}
    +
    \|s_-\|_{C^m(\Omega_{\omega,\kappa})}
    +
    \|\ell_\omega^{-1}\|_{C^m(\Omega_{\omega,\kappa})}
    \le C .
\end{equation}
Second, if \(f\in C^m(\overline\Om)\), then
\begin{equation}\label{eq:w-Cm-nonglancing-estimate}
    \|w_\omega^f\|_{C^m(Q_{\omega,\kappa})}
    \le
    C\|f\|_{C^m(\overline\Om)} .
\end{equation}
Here the \(C^m\)-norm on \(Q_{\omega,\kappa}\) is taken in the coordinates \((y,s)\in\omega^\perp\times\R\).  In particular, for every such non-glancing region,
\begin{equation}\label{eq:w-C2-nonglancing-estimate}
    \|D^2 w_\omega^f\|_{L^\infty(Q_{\omega,\kappa})}
    \le
    C\|f\|_{C^2(\overline\Om)} .
\end{equation}
Moreover, after increasing the constant if necessary,
\begin{equation}\label{eq:normal-derivative-nonglancing-estimate}
    \|\partial_\nu w_\omega^f\|_{C^m(\Gamma_{\omega,\kappa}^+\cup\Gamma_{\omega,\kappa}^-)}
    \le
    C\|f\|_{C^{m+1}(\overline\Om)} .
\end{equation}
\end{lemma}

\begin{proof}
We may choose \(\rho\in C^\infty(\mathbb R^n)\) such that
\[
\Omega=\{\rho<0\},
\qquad
\partial\Omega=\{\rho=0\},
\qquad
\nabla\rho\ne0\quad\text{on }\partial\Omega .
\]
For instance, one may take the signed distance function in a collar neighborhood of \(\partial\Omega\), with negative sign in \(\Omega\), and then extend it smoothly away from the boundary while preserving its sign. With this convention, \(\nabla\rho\) points in the outward normal direction on \(\partial\Omega\), and hence
\[
\nabla\rho=|\nabla\rho|\,\nu
\quad\text{on }\partial\Omega .
\]  
The endpoint functions satisfy
\[
    \rho(y+s_\pm(y)\omega)=0.
\]
For a vector \(e\in\omega^\perp\), differentiation gives
\begin{equation}\label{eq:s-pm-first-derivative}
    \partial_e s_\pm(y)
    =
    -
    \frac{\nabla\rho(x_\pm(y))\cdot e}
         {\nabla\rho(x_\pm(y))\cdot\omega}
    =
    -
    \frac{\nabla\rho(x_\pm(y))\cdot e}
         {|\nabla\rho(x_\pm(y))|\,\nu(x_\pm(y))\cdot\omega} .
\end{equation}
On \(\Omega_{\omega,\kappa}\), the denominator is bounded away from zero by a constant depending only on \(\Omega\) and \(\kappa\).  Repeated differentiation of the implicit equation gives the corresponding higher derivative bounds for \(s_\pm\), uniformly in \(\omega\).

It remains to justify that the inverse chord length is also uniformly bounded on these non-glancing sets.  We claim that there is \(c=c(\Omega,\kappa)>0\) such that
\begin{equation}\label{eq:uniform-chord-length-lower-bound}
    \ell_\omega(y)\ge c
    \qquad
    \text{for all }\omega\in\Sn,\ y\in\Omega_{\omega,\kappa}.
\end{equation}
If not, then there exist \(\omega_j\in\Sn\) and \(y_j\in\Omega_{\omega_j,\kappa}\) such that \(\ell_{\omega_j}(y_j)\to0\).  Passing to a subsequence, \(\omega_j\to\omega\), and the two endpoints \(x_{+,j}=x_+^{\omega_j}(y_j)\), \(x_{-,j}=x_-^{\omega_j}(y_j)\) converge to the same point \(x\in\bd\Om\), because \(|x_{+,j}-x_{-,j}|=\ell_{\omega_j}(y_j)\to0\).  The non-glancing inequalities give
\[
    \omega_j\cdot\nu(x_{+,j})\ge\kappa,
    \qquad
    -\omega_j\cdot\nu(x_{-,j})\ge\kappa .
\]
Letting \(j\to\infty\) yields simultaneously
\[
    \omega\cdot\nu(x)\ge\kappa,
    \qquad
    \omega\cdot\nu(x)\le-\kappa,
\]
which is impossible.  This proves \eqref{eq:uniform-chord-length-lower-bound}.  The estimates for \(\ell_\omega^{-1}\) and its derivatives now follow from the bounds for \(s_\pm\), the identity \(\ell_\omega=s_+-s_-\), and the lower bound \eqref{eq:uniform-chord-length-lower-bound}.

The estimate for \(w_\omega^f\) follows by differentiating the explicit formula \eqref{eq:w-explicit}.  Each derivative falls either on \(f(y+r\omega)\), on the endpoints \(s_\pm(y)\), or on \(\ell_\omega(y)^{-1}\).  The preceding bounds control the endpoint terms, while the derivatives of \(f(y+r\omega)\) are bounded by \(\|f\|_{C^m(\overline\Om)}\).  This proves \eqref{eq:w-Cm-nonglancing-estimate}, and \eqref{eq:w-C2-nonglancing-estimate} is the case \(m=2\).

Finally, by Lemma~\ref{lem:chord-parametrization}, the maps \(y\mapsto x_\pm(y)\) parametrize the non-glancing boundary pieces \(\Gamma_\omega^\pm\).  Since \(w_\omega^f(y,s_\pm(y))=0\), the function \(w_\omega^f\) vanishes identically on each of these boundary pieces, so \(\nabla w_\omega^f\) is normal to \(\bd\Omega\) there.  Hence
\[
    \partial_\nu w_\omega^f(x_\pm(y))
    =
    \frac{\partial_s w_\omega^f(y,s_\pm(y))}
         {\omega\cdot\nu(x_\pm(y))} .
\]
The denominator is bounded away from zero on \(\Gamma_{\omega,\kappa}^\pm\).  Differentiating this formula in the boundary coordinates \(y\) gives \eqref{eq:normal-derivative-nonglancing-estimate}.  The right hand side is written with \(\|f\|_{C^{m+1}}\) because the differentiated boundary quantity contains derivatives of \(\partial_s w_\omega^f\), hence derivatives of the chordwise solution one order beyond those used to control \(w_\omega^f\) itself.
\end{proof}

The non-glancing estimates above cannot be extended to glancing by differentiating the individual endpoint functions \(s_\pm\), because \eqref{eq:s-pm-first-derivative} contains the factor \((\omega\cdot\nu)^{-1}\).  This singularity is an artifact of the parametrization by the two separate roots.  The combinations entering the chordwise Dirichlet problem have additional cancellation.  The next lemma records the needed global statement for each fixed direction \(\omega\).

We shall use the following standard elementary fact.  If \(U\) is an open set and
\(H\in C^\infty(U\times(-\epsilon,\epsilon))\) is even in the second variable, then, after possibly shrinking \(\epsilon\), there is a function
\(\widetilde H\), smooth for \(\mu\ge0\) and smoothly extendable to a neighborhood of \(\mu=0\), such that
\[
    H(y,\lambda)=\widetilde H(y,\lambda^2).
\]
Equivalently, a smooth even function of \(\lambda\) is a smooth function of \(\lambda^2\).  %This is the one-dimensional invariant-function theorem for the action \(\lambda\mapsto-\lambda\), and may be proved by Taylor expansion in \(\lambda\), Borel's theorem for the even jets, and a smooth extension from the half-line.

\begin{lemma}\label{lem:global-glancing-smoothness}
Fix \(\omega\in\Sn\).  If \(f\in C^\infty(\overline\Omega)\), then the chordwise solution \(w_\omega^f\), initially defined for \(y\in\Omega_\omega\) and \(s_-(y)<s<s_+(y)\), extends to a function in \(C^\infty(\overline\Omega)\).  The chordwise quadratic barrier
\[
    b_\omega(y,s)=\frac12(s-s_-(y))(s-s_+(y))
\]
extends to a function in \(C^\infty(\overline\Omega)\) as well.  Moreover,
\[
    w_\omega^f=0,
    \qquad
    b_\omega=0
    \quad\text{on }\partial\Omega,
\]
and, in the coordinates \(x=y+s\omega\),
\[
    \partial_s^2w_\omega^f=f,
    \qquad
    \partial_s^2b_\omega=1
\]
throughout \(\Omega\), with the identities extending continuously to \(\overline\Omega\).  For every \(m\), the norms \(\|w_\omega^f\|_{C^m(\overline\Omega)}\) and \(\|b_\omega\|_{C^m(\overline\Omega)}\) are finite.  No uniformity in \(\omega\) is asserted in this global glancing statement.
\end{lemma}

\begin{proof}
Away from the glancing set this is the ordinary implicit-function argument already used in Lemma~\ref{lem:non-glancing-regularity}.  It remains to analyze a glancing point \(x_0=y_0+s_0\omega\in\partial\Omega\), so \(\omega\cdot\nu(x_0)=0\).  Let \(\rho\) be a smooth signed-distance defining function near \(x_0\), negative in \(\Omega\).  At \(x_0\),
\[
    \partial_s\rho(x_0)=\nabla\rho(x_0)\cdot\omega=0.
\]
Since \(\omega\) is tangent to \(\partial\Omega\) at \(x_0\) and \(\Omega\) is uniformly convex, the second fundamental form is positive definite with respect to the outward normal.  For the signed-distance defining function normalized by \(\rho<0\) in \(\Omega\), this gives the sign
\[
    \partial_s^2\rho(x_0)=D^2\rho(x_0)[\omega,\omega]>0 .
\]
By the one-dimensional Malgrange preparation theorem for smooth functions \cite{Malgrange1966}, after shrinking the coordinate neighborhood there are smooth functions \(a(y)\), \(\beta(y)\), and a smooth non-vanishing factor \(h(y,s)\) such that, with \(z=s-a(y)\),
\[
    \rho(y,s)=h(y,s)\bigl(z^2-\beta(y)\bigr).
\]
Since \(2h(y_0,s_0)=\partial_s^2\rho(x_0)>0\), we may shrink the neighborhood so that \(h>0\).  Hence, because \(\rho<0\) in \(\Omega\), the part of \(\Omega\) in this neighborhood is represented by
\[
    z^2<\beta(y).
\]
Thus the two chord endpoints are
\[
    s_\pm(y)=a(y)\pm\sqrt{\beta(y)}
\]
whenever \(\beta(y)>0\).  Consequently
\[
    b_\omega(y,s)
    =\frac12\bigl(z^2-\beta(y)\bigr)
\]
in this neighborhood, and hence \(b_\omega\) is smooth through the glancing point.

We now treat \(w_\omega^f\).  Set
\[
    F(y,z)=f(y+(a(y)+z)\omega).
\]
For fixed \(y\), let
\[
    P(y,z)=\int_0^z (z-r)F(y,r)\,dr,
    \qquad
    \partial_z^2P=F.
\]
For \(\lambda>0\), define
\[
    A(y,\lambda)=\frac{P(y,-\lambda)-P(y,\lambda)}{2\lambda},
    \qquad
    B(y,\lambda)=-\frac{P(y,\lambda)+P(y,-\lambda)}2 .
\]
The functions \(A\) and \(B\) are smooth and even in \(\lambda\) after extending across \(\lambda=0\).  By the even-factorization fact stated above, after possibly shrinking the neighborhood, there are smooth functions \(\widetilde A(y,\mu)\) and \(\widetilde B(y,\mu)\), defined for \(\mu\ge0\) near \(0\) and smoothly extendable across \(\mu=0\), such that
\[
    A(y,\lambda)=\widetilde A(y,\lambda^2),
    \qquad
    B(y,\lambda)=\widetilde B(y,\lambda^2).
\]
Define
\[
    W(y,z)=P(y,z)+\widetilde A(y,\beta(y))z+\widetilde B(y,\beta(y)).
\]
Then \(W\) is smooth in \((y,z)\).  If \(\beta(y)>0\) and \(\lambda=\sqrt{\beta(y)}\), then
\[
    W(y,\pm\lambda)=0,
    \qquad
    \partial_z^2W(y,z)=F(y,z).
\]
By uniqueness for the one-dimensional Dirichlet problem on the chord, \(W(y,s-a(y))\) agrees with \(w_\omega^f(y,s)\) for the interior chords in the neighborhood.  This proves smooth extension through the glancing point.  On overlaps, the local extensions agree with the original chordwise solution on the interior chords, and therefore agree with each other by continuity.  A finite covering of the compact glancing set, together with the non-glancing smoothness, proves the global assertion on \(\overline\Omega\).
\end{proof}

\subsection{The determinant expansion for approximate solutions}\label{subsec:determinant-expansion}

Set
\[
    U_t=t\phi_\omega+t^{1-n}w_\omega^f.
\]
At non-glancing boundary points this satisfies \(U_t=t\phi_\omega\).  At glancing boundary points this equality is understood using the continuous extension of \(w_\omega^f\) described after Lemma~\ref{lem:chordwise-C0-bounds}.  The following lemma gives the precise determinant expansion behind the formal statement \eqref{eq:det-expansion}.

Fix an orthonormal coordinate system \((y,s)\in\omega^\perp\times\R\), where \(s=x\cdot\omega\).  For a function \(w=w(y,s)\), write
\[
    B=D_y^2w,
    \qquad
    z=D_y\partial_s w,
    \qquad
    a=\partial_s^2 w,
    \qquad
    q=t^{-n}.
\]
Here \(B\) is an \((n-1)\times(n-1)\) symmetric matrix, \(z\in\R^{n-1}\), and \(a\in\R\).

\begin{lemma}\label{lem:det-exact-schur}
Let \(w\in C^2\) on an open set, and suppose that \(I+qB\) is invertible at the point under consideration.  Then
\begin{equation}\label{eq:det-exact-schur}
\begin{aligned}
    \det D^2(t\phi_\omega+t^{1-n}w)
    &=
    \det(I+qB)
    \left(
        a-q\left\langle (I+qB)^{-1}z,z\right\rangle
    \right).
\end{aligned}
\end{equation}
In particular, the identity holds for all sufficiently large \(t\) on any set where \(D_y^2w\) is bounded.
\end{lemma}

\begin{proof}
In the coordinates \(x=y+s\omega\), one has
\begin{equation}\label{eq:hessian-block-form}
    D^2(t\phi_\omega+t^{1-n}w)
    =
    \begin{pmatrix}
        tI_{n-1}+t^{1-n}B
        &
        t^{1-n}z \\
        t^{1-n}z^T
        &
        t^{1-n}a
    \end{pmatrix}.
\end{equation}
The upper-left block is
\[
    A_t=tI_{n-1}+t^{1-n}B=t(I+qB).
\]
Assuming \(A_t\) is invertible, the Schur-complement formula gives
\begin{align*}
    \det D^2(t\phi_\omega+t^{1-n}w)
    &=
    \det A_t
    \left(
        t^{1-n}a
        -
        t^{2-2n}z^T A_t^{-1}z
    \right) \\
    &=
    t^{n-1}\det(I+qB)
    \left(
        t^{1-n}a
        -
        t^{2-2n}t^{-1}z^T(I+qB)^{-1}z
    \right) \\
    &=
    \det(I+qB)
    \left(
        a-q\left\langle (I+qB)^{-1}z,z\right\rangle
    \right),
\end{align*}
because \(q=t^{-n}\).  This proves the formula.
\end{proof}

\begin{lemma}\label{lem:det-first-order}
Let \(M\ge1\).  There exists a constant \(C=C(n,M)\) such that the following holds.  Suppose
\[
    |a|+|z|+\|B\|\le M,
    \qquad
    t^{-n}\|B\|\le \frac12.
\]
Then
\begin{equation}\label{eq:det-first-order}
\begin{aligned}
    \det D^2(t\phi_\omega+t^{1-n}w)
    &=
    a
    +
    t^{-n}\left(a\tr B-|z|^2\right)
    +
    R_t,
\end{aligned}
\end{equation}
where
\begin{equation}\label{eq:det-remainder-bound}
    |R_t|\le C t^{-2n}.
\end{equation}
Equivalently,
\begin{equation}\label{eq:det-first-order-w}
\begin{aligned}
    \det D^2(t\phi_\omega+t^{1-n}w)
    &=
    \partial_s^2 w \\
    &\quad
    +
    t^{-n}\left[
        (\partial_s^2w)\tr(D_y^2w)
        -
        |D_y\partial_s w|^2
    \right]
    +O(t^{-2n}).
\end{aligned}
\end{equation}
The constant in the \(O(t^{-2n})\) term depends only on \(n\) and on the stated bound for the second derivatives of \(w\).
\end{lemma}

\begin{proof}
By Lemma~\ref{lem:det-exact-schur},
\[
    \det D^2(t\phi_\omega+t^{1-n}w)
    =
    \det(I+qB)
    \left(a-q\left\langle (I+qB)^{-1}z,z\right\rangle\right).
\]
For \(q\|B\|\le1/2\), the elementary matrix expansions give
\[
    \det(I+qB)=1+q\tr B+O(q^2),
    \qquad
    (I+qB)^{-1}=I+O(q),
\]
with constants depending only on \(n\) and \(M\).  Hence
\[
    \left\langle (I+qB)^{-1}z,z\right\rangle
    =
    |z|^2+O(q).
\]
Substituting these estimates into the exact formula gives
\[
\begin{aligned}
    \det D^2(t\phi_\omega+t^{1-n}w)
    &=
    (1+q\tr B+O(q^2))
    (a-q|z|^2+O(q^2)) \\
    &=
    a+q(a\tr B-|z|^2)+O(q^2).
\end{aligned}
\]
Since \(q=t^{-n}\), this is \eqref{eq:det-first-order}--\eqref{eq:det-first-order-w}.
\end{proof}

\begin{corollary}\label{cor:approx-source-equation}
Let \(w=w_\omega^f\) satisfy
\[
    \partial_s^2w=f
\]
on a region where the second derivatives of \(w\) are bounded.  Then the approximate solution
\[
    U_t=t\phi_\omega+t^{1-n}w_\omega^f
\]
satisfies
\begin{equation}\label{eq:approx-source}
\begin{aligned}
    \det D^2U_t
    &=
    f
    +
    t^{-n}\left[
        f\tr(D_y^2w_\omega^f)
        -
        |D_y\partial_s w_\omega^f|^2
    \right]
    +O(t^{-2n}).
\end{aligned}
\end{equation}
In particular,
\begin{equation}\label{eq:approx-source-short}
    \det D^2U_t=f+O(t^{-n})
\end{equation}
locally uniformly on every region where \(\|D^2w_\omega^f\|\) is bounded.
\end{corollary}

\begin{proof}
Apply Lemma~\ref{lem:det-first-order} with \(w=w_\omega^f\) and use \(\partial_s^2w_\omega^f=f\).
\end{proof}

The scaling in \eqref{eq:det-expansion} is as follows.  The leading transverse block has determinant \(t^{n-1}\), while the longitudinal second derivative of the correction has size \(t^{1-n}\partial_s^2w\).  Their product is order one.  The first relative perturbation of the transverse block is
\[
    t^{-1}t^{1-n}D_y^2w=t^{-n}D_y^2w,
\]
which is why the first error term is of order \(t^{-n}\).  No term of order \(t^{1-n}\) appears in the determinant expansion.  The off-diagonal blocks contribute only through the Schur complement: before multiplication by \(\det A_t\), their contribution is \(t^{2-2n}z^TA_t^{-1}z=O(t^{1-2n})\), and after multiplication by \(\det A_t=O(t^{n-1})\) this is again of order \(t^{-n}\).

\subsection{Convexity of the approximate solution}\label{subsec:convexity-approximate-solutions}

We next record the convexity estimate for the approximate profile.  This is a local statement in a coordinate tube on which the second derivatives of the chordwise correction are bounded.  By Lemma~\ref{lem:non-glancing-regularity}, this applies uniformly on every non-glancing tube \(Q_{\omega,\kappa}\).

The estimate is slightly more general than needed for \(U_t\), because the comparison argument will later use perturbations of size \(t^{1-2n}\).  Let \(w,\eta\in C^2(Q)\), where \(Q\) is a region written in coordinates \((y,s)\in\omega^\perp\times\R\).  For \(|\lambda|\le L\), set
\begin{equation}\label{eq:perturbed-approx-profile}
    V_{t,\lambda}
    =
    t\phi_\omega+t^{1-n}w+\lambda t^{1-2n}\eta
    =
    t\phi_\omega+t^{1-n}\bigl(w+\lambda t^{-n}\eta\bigr).
\end{equation}

\begin{lemma}\label{lem:schur-convexity-criterion}
Let \(c_0>0\), \(M<\infty\), and \(L<\infty\).  Suppose that on \(Q\)
\begin{equation}\label{eq:convexity-hypotheses-general}
    \partial_s^2 w\ge c_0,
    \qquad
    \|D^2w\|+\|D^2\eta\|\le M .
\end{equation}
Then there is \(T=T(n,c_0,M,L)\) such that, for every \(t\ge T\) and every \(|\lambda|\le L\),
\[
    D^2V_{t,\lambda}>0
    \quad\text{on }Q.
\]
More precisely, writing
\[
    W_{t,\lambda}=w+\lambda t^{-n}\eta,
\]
one has
\begin{equation}\label{eq:convexity-transverse-block}
    I+t^{-n}D_y^2W_{t,\lambda}
    \ge \frac12 I
\end{equation}
and
\begin{equation}\label{eq:convexity-schur-lower}
    \partial_s^2W_{t,\lambda}
    -t^{-n}
    \left\langle
        (I+t^{-n}D_y^2W_{t,\lambda})^{-1}D_y\partial_sW_{t,\lambda},
        D_y\partial_sW_{t,\lambda}
    \right\rangle
    \ge \frac{c_0}{2} .
\end{equation}
\end{lemma}

\begin{proof}
Put \(q=t^{-n}\).  The Hessian of \(V_{t,\lambda}=t\phi_\omega+t^{1-n}W_{t,\lambda}\) has block form
\begin{equation}\label{eq:hessian-block-form-convexity}
    D^2V_{t,\lambda}
    =
    \begin{pmatrix}
        t(I+qD_y^2W_{t,\lambda})
        &
        t^{1-n}D_y\partial_sW_{t,\lambda} \\
        t^{1-n}(D_y\partial_sW_{t,\lambda})^T
        &
        t^{1-n}\partial_s^2W_{t,\lambda}
    \end{pmatrix}.
\end{equation}
The assumptions imply, for \(t\ge1\) and \(|\lambda|\le L\),
\[
    \|D^2W_{t,\lambda}\|
    \le
    M+LM
    =:M_1 .
\]
If \(qM_1\le1/2\), then \eqref{eq:convexity-transverse-block} holds and the transverse block is positive definite.

It remains to check the Schur complement.  Since
\[
    \partial_s^2W_{t,\lambda}
    =
    \partial_s^2w+\lambda q\partial_s^2\eta,
\]
we have
\[
    \partial_s^2W_{t,\lambda}
    \ge
    c_0-LqM .
\]
Also, when \(qM_1\le1/2\),
\[
    \|(I+qD_y^2W_{t,\lambda})^{-1}\|\le2,
    \qquad
    |D_y\partial_sW_{t,\lambda}|\le M_1 .
\]
Therefore
\[
\begin{aligned}
    &\partial_s^2W_{t,\lambda}
    -q
    \left\langle
        (I+qD_y^2W_{t,\lambda})^{-1}D_y\partial_sW_{t,\lambda},
        D_y\partial_sW_{t,\lambda}
    \right\rangle \\
    &\qquad\ge
    c_0-LqM-2qM_1^2 .
\end{aligned}
\]
For \(t\) large enough, the right-hand side is at least \(c_0/2\).  The positivity of the transverse block and of the Schur complement implies \(D^2V_{t,\lambda}>0\).
\end{proof}

\begin{corollary}\label{cor:model-profile-convexity}
Let \(f\in C^\infty(\overline\Omega)\) satisfy \(f\ge c_f>0\), and fix \(\omega\in\Sn\).  Let \(\eta\in C^2(\overline\Omega)\).  For every fixed \(L<\infty\), there is \(T=T(f,\omega,\eta,L)\) such that
\[
    t\phi_\omega+t^{1-n}w_\omega^f+\lambda t^{1-2n}\eta
\]
is strictly convex in \(\Omega\) whenever \(t\ge T\) and \(|\lambda|\le L\).  In particular, for every fixed \(C<\infty\), the functions
\[
    U_t=t\phi_\omega+t^{1-n}w_\omega^f,
    \qquad
    U_t\pm C t^{1-2n}b_\omega
\]
are strictly convex in \(\Omega\) for all sufficiently large \(t\).
\end{corollary}

\begin{proof}
By Lemma~\ref{lem:global-glancing-smoothness}, the functions \(w_\omega^f\) and \(b_\omega\) have bounded second derivatives on \(\overline\Omega\) for the fixed direction \(\omega\).  The chordwise equation gives
\[
    \partial_s^2w_\omega^f=f\ge c_f.
\]
The assertion follows directly from Lemma~\ref{lem:schur-convexity-criterion}, with \(Q=\Omega\) written in the coordinates \(x=y+s\omega\).  For the final statement, take \(\eta=b_\omega\).
\end{proof}

\subsection{Global comparison barriers and the \texorpdfstring{\(C^0\)}{C0} asymptotic}\label{subsec:barrier-construction}

We now compare the true solution \(u_{t,\omega}^f\) with
\[
    U_t=t\phi_\omega+t^{1-n}w_\omega^f .
\]
The global smoothness statement in Lemma~\ref{lem:global-glancing-smoothness} allows the chordwise barriers to be used on all of \(\Omega\), including near glancing.

Define
\begin{equation}\label{eq:barrier-b-explicit}
    b_\omega(y,s)
    =
    \frac12\bigl(s-s_-(y)\bigr)\bigl(s-s_+(y)\bigr).
\end{equation}
Then
\begin{equation}\label{eq:barrier-b}
    \partial_s^2 b_\omega=1,
    \qquad
    b_\omega=0\quad\text{on }\partial\Omega,
\end{equation}
and
\begin{equation}\label{eq:b-negative}
    b_\omega\le0
    \quad\text{on each chord},
    \qquad
    |b_\omega(y,s)|\le \frac18 \ell_\omega(y)^2
    \le \frac18\operatorname{diam}(\Omega)^2 .
\end{equation}

We will use Theorem~\ref{thm:MA-comparison} in the following classical form.  If \(D\subset\R^n\) is bounded and \(v,u\in C^2(\overline D)\) are convex, with
\[
    \det D^2v\ge \det D^2u\quad\text{in }D,
    \qquad
    v\le u\quad\text{on }\partial D,
\]
then \(v\le u\) in \(D\).  Thus, for equal boundary values, the convex function with the larger Monge--Ampere determinant lies lower.

\begin{proposition}\label{prop:comparison-barriers}
Let \(f\in C^\infty(\overline\Omega)\) satisfy \(f\ge c_f>0\), and fix \(\omega\in\Sn\).  There exist constants \(C_*>0\) and \(T<\infty\), depending on \(f\), \(\omega\), and \(\Omega\), such that, for every \(t\ge T\), the functions
\begin{equation}\label{eq:lower-upper-barriers}
    V_t^{\rm low}
    =U_t+C_*t^{1-2n}b_\omega,
    \qquad
    V_t^{\rm up}
    =U_t-C_*t^{1-2n}b_\omega
\end{equation}
are strictly convex in \(\Omega\), agree with \(t\phi_\omega\) on \(\partial\Omega\), and satisfy
\begin{equation}\label{eq:barrier-order-around-U}
    V_t^{\rm low}\le U_t\le V_t^{\rm up}
    \quad\text{in }\Omega,
\end{equation}
\begin{equation}\label{eq:barrier-det-inequalities}
    \det D^2V_t^{\rm low}\ge f,
    \qquad
    \det D^2V_t^{\rm up}\le f
    \quad\text{in }\Omega .
\end{equation}
Consequently,
\begin{equation}\label{eq:global-C0-asymptotic}
    U_t+C_*t^{1-2n}b_\omega
    \le
    u_{t,\omega}^f
    \le
    U_t-C_*t^{1-2n}b_\omega
    \quad\text{in }\Omega,
\end{equation}
and hence
\begin{equation}\label{eq:global-C0-asymptotic-bound}
    \|u_{t,\omega}^f-U_t\|_{L^\infty(\Omega)}
    \le
    \frac{C_*}{8}\operatorname{diam}(\Omega)^2 t^{1-2n}.
\end{equation}
\end{proposition}

\begin{proof}
The boundary equality follows from \(w_\omega^f=b_\omega=0\) on \(\partial\Omega\).  The order \eqref{eq:barrier-order-around-U} follows from \(b_\omega\le0\).  Strict convexity for large \(t\) follows from Corollary~\ref{cor:model-profile-convexity} with \(\eta=b_\omega\) and \(|\lambda|=C_*\).

It remains to prove the determinant inequalities.  Put \(q=t^{-n}\) and write
\[
    E_\omega^f
    =
    f\,\tr(D_y^2w_\omega^f)
    -|D_y\partial_s w_\omega^f|^2 .
\]
By Lemma~\ref{lem:global-glancing-smoothness},
\begin{equation}\label{eq:E-bound-global}
    A:=\|E_\omega^f\|_{L^\infty(\Omega)}<\infty .
\end{equation}
For \(\sigma\in\{+1,-1\}\), define
\[
    W_t^\sigma=w_\omega^f+\sigma C_*q b_\omega,
    \qquad
    V_t^\sigma=t\phi_\omega+t^{1-n}W_t^\sigma .
\]
Thus \(V_t^+=V_t^{\rm low}\) and \(V_t^-=V_t^{\rm up}\).  Since
\[
    \partial_s^2W_t^\sigma=f+\sigma C_*q,
\]
and since \(D^2w_\omega^f\) and \(D^2b_\omega\) are bounded on \(\Omega\), Lemma~\ref{lem:det-first-order} gives, uniformly in \(\Omega\),
\begin{equation}\label{eq:det-barrier-expansion}
    \det D^2V_t^\sigma
    =
    f
    +q\left(\sigma C_*+E_\omega^f\right)
    +O_{C_*}(q^2).
\end{equation}
Choose \(C_*\ge 2A+2\).  After increasing \(T\), the remainder in \eqref{eq:det-barrier-expansion} is bounded in absolute value by \(q\) for \(t\ge T\).  For \(\sigma=+1\),
\[
    \det D^2V_t^{\rm low}
    \ge
    f+q(C_*-A)-q
    \ge f.
\]
For \(\sigma=-1\),
\[
    \det D^2V_t^{\rm up}
    \le
    f-q(C_*-A)+q
    \le f.
\]
This proves \eqref{eq:barrier-det-inequalities}.

On \(\partial\Omega\), both barriers equal \(u_{t,\omega}^f=t\phi_\omega\).  Since
\[
    \det D^2V_t^{\rm low}\ge f=\det D^2u_{t,\omega}^f,
\]
the comparison principle in the orientation recalled above gives
\[
    V_t^{\rm low}\le u_{t,\omega}^f
    \quad\text{in }\Omega.
\]
Similarly,
\[
    \det D^2u_{t,\omega}^f=f\ge \det D^2V_t^{\rm up},
\]
and the same comparison principle applied to the pair \((u_{t,\omega}^f,V_t^{\rm up})\) gives
\[
    u_{t,\omega}^f\le V_t^{\rm up}
    \quad\text{in }\Omega.
\]
This proves \eqref{eq:global-C0-asymptotic}.  The norm bound \eqref{eq:global-C0-asymptotic-bound} follows from \eqref{eq:b-negative}.
\end{proof}

\subsection{Boundary normal derivatives}\label{subsec:boundary-normal-derivatives}

We now derive the boundary statement \eqref{eq:DN-asymptotic} on compact subsets of the non-glancing boundary.  The global \(C^0\) estimate from Proposition~\ref{prop:comparison-barriers} controls the artificial boundary of small one-sided collars.  Inside those collars we use barriers with a larger but still lower-order profile perturbation.

For reference, throughout this part of the argument \(U_t=t\phi_\omega+t^{1-n}w_\omega^f\), \(q=t^{-n}\), and \(b_\omega\) denotes the chordwise quadratic barrier from \eqref{eq:barrier-b-explicit}.  The auxiliary functions \(G_{x_0}\) and collars \(D_{x_0}^\pm\) are introduced below at the boundary point \(x_0\) where the normal derivative is estimated.

Choose
\[
    \varepsilon_t=t^{-n/2},
    \qquad
    \eta_t=t^{1-n}\varepsilon_t=t^{1-n-n/2}.
\]
Then \(\varepsilon_t\to0\), while \(t^n\varepsilon_t\to\infty\).  Thus the determinant correction generated by the local barriers dominates the \(O(t^{-n})\) determinant defect of \(U_t\), and their normal derivative error is \(\eta_t=o(t^{1-n})\).

For a boundary point \(x_0\in\Gamma_\omega\), set
\begin{equation}\label{eq:p-x0-def}
    y_0=P_{\omega^\perp}x_0,
    \qquad
    p_{x_0}(x)=|P_{\omega^\perp}x-y_0|^2.
\end{equation}
Thus \(p_{x_0}\) is independent of the \(s\)-variable and
\begin{equation}\label{eq:p-x0-normal-zero}
    p_{x_0}(x_0)=0,
    \qquad
    \partial_\nu p_{x_0}(x_0)=0.
\end{equation}
Define
\begin{equation}\label{eq:G-x0-def}
    G_{x_0}= -b_\omega+p_{x_0}.
\end{equation}
Since \(b_\omega\le0\), the function \(G_{x_0}\) is nonnegative, and
\begin{equation}\label{eq:G-ss}
    \partial_s^2G_{x_0}=-1.
\end{equation}

We shall also use the sign of the outward normal derivative of \(b_\omega\) on non-glancing boundary pieces.  For \(y\in\Omega_\omega\),
\begin{equation}\label{eq:b-normal-step8}
    \partial_\nu b_\omega(x_\pm(y))
    =
    \frac{\ell_\omega(y)}{2|\omega\cdot\nu(x_\pm(y))|}
    >0.
\end{equation}
Indeed, \(b_\omega\) vanishes on each non-glancing boundary piece, so \(\nabla b_\omega\) is normal there.  Moreover,
\[
    \partial_s b_\omega(y,s_+(y))=\frac12\ell_\omega(y),
    \qquad
    \partial_s b_\omega(y,s_-(y))=-\frac12\ell_\omega(y),
\]
and the endpoint signs \(\omega\cdot\nu(x_+)>0\), \(\omega\cdot\nu(x_-)<0\) give \eqref{eq:b-normal-step8}.

\begin{proposition}\label{prop:boundary-normal-from-C0}
Let \(f\in C^\infty(\overline\Omega)\) satisfy \(f\ge c_f>0\).  Fix \(\omega\in\Sn\), and let \(K\Subset\Gamma_\omega\).  Then there are constants \(C_K<\infty\) and \(T_K<\infty\) such that, for all \(t\ge T_K\),
\begin{equation}\label{eq:boundary-normal-rate-from-C0}
    \sup_{x\in K}
    \left|
        \partial_\nu u_{t,\omega}^f(x)
        -\partial_\nu U_t(x)
    \right|
    \le
    C_K t^{1-n-n/2}.
\end{equation}
In particular,
\begin{equation}\label{eq:boundary-normal-asymptotic-proved}
    \partial_\nu u_{t,\omega}^f
    =
    t\partial_\nu\phi_\omega
    +
    t^{1-n}\partial_\nu w_\omega^f
    +
    o(t^{1-n})
\end{equation}
locally uniformly on \(\Gamma_\omega\).
\end{proposition}

\begin{proof}
We prove the estimate pointwise with constants uniform for \(x_0\in K\).  Write
\[
    K_+=K\cap\Gamma_\omega^+,
    \qquad
    K_-=K\cap\Gamma_\omega^- .
\]
Since \(K\Subset\Gamma_\omega\), there are \(\kappa>0\) and an open set \(Y\Subset\Omega_\omega\) such that, with
\[
    Y_+=P_{\omega^\perp}(K_+),
    \qquad
    Y_-=P_{\omega^\perp}(K_-),
\]
one has \(Y_+\cup Y_-\subset Y\) and
\[
    \omega\cdot\nu(x_+(y))\ge\kappa,
    \qquad
    -\omega\cdot\nu(x_-(y))\ge\kappa
    \qquad \text{for all }y\in Y .
\]
Lemma~\ref{lem:non-glancing-regularity} gives uniform bounds for the endpoint functions and for \(\partial_\nu b_\omega\) on this non-glancing set.  Since \(\overline Y\subset\Omega_\omega\) and \(\ell_\omega\) is continuous and positive on \(\Omega_\omega\), there is \(\ell_0>0\) such that
\[
    \ell_\omega(y)=s_+(y)-s_-(y)\ge \ell_0
    \qquad \text{for all }y\in Y .
\]
Choose \(\rho>0\) so small that \(\{y\in\omega^\perp:|y-y_0|<\rho\}\subset Y\) for every \(y_0\in Y_+\cup Y_-\), and then choose \(0<\delta<\ell_0/4\).  Reducing \(\rho\) and \(\delta\), if necessary, we may also assume that the following collars remain inside this non-glancing tube.  If \(x_0\in K_+\), set \(y_0=P_{\omega^\perp}x_0\) and
\begin{equation}\label{eq:plus-boundary-collar}
    D_{x_0}^+
    =
    \{y+s\omega: |y-y_0|<\rho,\ s_+(y)-\delta<s<s_+(y)\}.
\end{equation}
If \(x_0\in K_-\), set
\begin{equation}\label{eq:minus-boundary-collar}
    D_{x_0}^-
    =
    \{y+s\omega: |y-y_0|<\rho,\ s_-(y)<s<s_-(y)+\delta\}.
\end{equation}
We write \(D_{x_0}\) for either collar.  The Monge--Ampere comparison principle for convex solutions is valid on bounded domains, so the corners of these collars cause no difficulty; one may also justify the same step by applying the comparison principle on a smooth exhaustion and passing to the limit.

Let \(\varepsilon_t\) and \(\eta_t\) be as above.  Define
\begin{equation}\label{eq:boundary-local-barriers}
    L_{t,x_0}=U_t-\eta_tG_{x_0},
    \qquad
    W_{t,x_0}=U_t+\eta_tG_{x_0}.
\end{equation}
Since \(G_{x_0}\ge0\), these barriers are ordered around \(U_t\).  On the physical boundary part of \(D_{x_0}\), one has \(w_\omega^f=0\) and \(b_\omega=0\), so
\[
    L_{t,x_0}=t\phi_\omega-\eta_t p_{x_0}
    \le
    t\phi_\omega
    =u_{t,\omega}^f
    \le
    t\phi_\omega+\eta_t p_{x_0}
    =W_{t,x_0}.
\]
At the distinguished point \(x_0\), both inequalities are equalities.

On the artificial boundary of \(D_{x_0}\), the function \(G_{x_0}\) is bounded below by a positive constant independent of \(x_0\).  On the tangential face \(|y-y_0|=\rho\), one has \(G_{x_0}\ge p_{x_0}=\rho^2\).  On the inner face, for either sign,
\[
    -b_\omega=\frac12\delta(\ell_\omega(y)-\delta)
    \ge \frac12\delta(\ell_0-\delta).
\]
Thus there exists \(c_0>0\) such that
\begin{equation}\label{eq:G-positive-artificial-boundary}
    G_{x_0}\ge c_0
    \quad\text{on }\partial D_{x_0}\setminus\partial\Omega,
\end{equation}
for every \(x_0\in K\).

By Proposition~\ref{prop:comparison-barriers}, there is \(M<\infty\) such that
\begin{equation}\label{eq:global-C0-input-for-boundary}
    \|u_{t,\omega}^f-U_t\|_{L^\infty(\Omega)}\le M t^{1-2n}
\end{equation}
for all sufficiently large \(t\).  Since
\[
    \frac{\eta_t}{t^{1-2n}}=t^{n/2}\to\infty,
\]
\eqref{eq:G-positive-artificial-boundary} and \eqref{eq:global-C0-input-for-boundary} imply, for all sufficiently large \(t\),
\[
    L_{t,x_0}\le u_{t,\omega}^f\le W_{t,x_0}
    \quad\text{on }\partial D_{x_0}.
\]

We next check the determinant inequalities.  The profiles in \eqref{eq:boundary-local-barriers} are
\[
    L_{t,x_0}
    =
    t\phi_\omega+t^{1-n}(w_\omega^f-\varepsilon_tG_{x_0}),
    \qquad
    W_{t,x_0}
    =
    t\phi_\omega+t^{1-n}(w_\omega^f+\varepsilon_tG_{x_0}).
\]
On the chosen collars, \(G_{x_0}\) has uniformly bounded second derivatives and \(\partial_s^2G_{x_0}=-1\).  Lemma~\ref{lem:det-first-order} gives, uniformly for \(x\in D_{x_0}\) and \(x_0\in K\),
\begin{equation}\label{eq:boundary-barrier-det-expansion}
\begin{aligned}
    \det D^2L_{t,x_0}&=f+\varepsilon_t+O(t^{-n}),\\
    \det D^2W_{t,x_0}&=f-\varepsilon_t+O(t^{-n}).
\end{aligned}
\end{equation}
Since \(\varepsilon_t=t^{-n/2}\) and \(t^{-n}=o(\varepsilon_t)\), after increasing the lower threshold for \(t\),
\begin{equation}\label{eq:boundary-barrier-det-inequalities}
    \det D^2L_{t,x_0}\ge f,
    \qquad
    \det D^2W_{t,x_0}\le f
    \quad\text{in }D_{x_0}.
\end{equation}
The same uniform \(C^2\)-bounds and the Schur-complement convexity argument of Lemma~\ref{lem:schur-convexity-criterion} show that both \(L_{t,x_0}\) and \(W_{t,x_0}\) are strictly convex for all sufficiently large \(t\).  Indeed, their longitudinal profile second derivatives are \(f+\varepsilon_t\) and \(f-\varepsilon_t\), both bounded below by \(c_f/2\) for large \(t\), while the transverse block remains a small relative perturbation of \(tI_{n-1}\).

The comparison principle, with the orientation recalled above, now gives
\begin{equation}\label{eq:boundary-barrier-comparison}
    L_{t,x_0}\le u_{t,\omega}^f\le W_{t,x_0}
    \quad\text{in }D_{x_0}.
\end{equation}
Indeed, the lower barrier has determinant at least \(f=\det D^2u_{t,\omega}^f\), so it lies below \(u_{t,\omega}^f\); the upper barrier has determinant at most \(f\), so \(u_{t,\omega}^f\) lies below it.

At \(x_0\), the three functions in \eqref{eq:boundary-barrier-comparison} have the same boundary value.  Since \(u_{t,\omega}^f-L_{t,x_0}\ge0\) and \(W_{t,x_0}-u_{t,\omega}^f\ge0\) in the one-sided collar, their outward normal derivatives at \(x_0\) are nonpositive.  Hence
\begin{equation}\label{eq:normal-derivative-squeeze}
    \partial_\nu W_{t,x_0}(x_0)
    \le
    \partial_\nu u_{t,\omega}^f(x_0)
    \le
    \partial_\nu L_{t,x_0}(x_0).
\end{equation}
Using \eqref{eq:p-x0-normal-zero} and \(G_{x_0}=-b_\omega+p_{x_0}\), we get
\[
    \partial_\nu G_{x_0}(x_0)=-\partial_\nu b_\omega(x_0).
\]
Therefore
\[
    \partial_\nu L_{t,x_0}(x_0)
    =
    \partial_\nu U_t(x_0)
    +\eta_t\partial_\nu b_\omega(x_0),
    \qquad
    \partial_\nu W_{t,x_0}(x_0)
    =
    \partial_\nu U_t(x_0)
    -\eta_t\partial_\nu b_\omega(x_0).
\]
By \eqref{eq:b-normal-step8}, \(\partial_\nu b_\omega(x_0)>0\), and the non-glancing estimates give
\[
    \sup_K |\partial_\nu b_\omega|\le C_K .
\]
Thus \eqref{eq:normal-derivative-squeeze} implies
\[
    \left|
        \partial_\nu u_{t,\omega}^f(x_0)
        -\partial_\nu U_t(x_0)
    \right|
    \le C_K\eta_t
    =C_Kt^{1-n-n/2},
\]
uniformly for \(x_0\in K\).  Finally,
\[
    \partial_\nu U_t
    =
    t\partial_\nu\phi_\omega
    +
    t^{1-n}\partial_\nu w_\omega^f,
\]
and \(t^{1-n-n/2}=o(t^{1-n})\).  This proves the proposition.
\end{proof}

\begin{proof}[Proof of Theorem~\ref{thm:large-data-asymptotic}]
The estimate \eqref{eq:bulk-asymptotic} follows from \eqref{eq:global-C0-asymptotic-bound} in Proposition~\ref{prop:comparison-barriers}, after absorbing the fixed factor \(\operatorname{diam}(\Omega)^2/8\) into the constant.  The estimate \eqref{eq:DN-asymptotic} follows from \eqref{eq:boundary-normal-rate-from-C0} in Proposition~\ref{prop:boundary-normal-from-C0}, together with
\[
    \partial_\nu U_t
    =t\partial_\nu\phi_\omega+t^{1-n}\partial_\nu w_\omega^f .
\]
\end{proof}

The comparison argument uses global barriers for each fixed direction \(\omega\).  The only non-glancing restriction left in the boundary derivative statement is intrinsic to the recovery formula: at glancing points the factor \(\omega\cdot\nu\) vanishes, so the endpoint identity will not be used there.  The estimates above are proved for each fixed \(\omega\); the constants are not claimed to be uniform as \(\omega\) varies.

\section{From boundary asymptotics to X-ray data}\label{sec:recovery-line-integrals}

We now show how the boundary asymptotic proved in Theorem~\ref{thm:large-data-asymptotic} determines chord integrals of the source.  This is the point at which the Dirichlet-to-Neumann map enters the recovery of the source.

\subsection{Boundary limits determined by the Dirichlet-to-Neumann map}\label{subsec:boundary-limits}

For \(f\in\mathcal F\), \(\omega\in\Sn\), and \(x\in\Gamma_\omega\), define
\begin{equation}\label{eq:L-omega}
    L_\omega^f(x)
    =
    \lim_{t\to\infty}
    t^{n-1}
    \left[
        \DN_f(t\phi_\omega|_{\bd\Om})(x)
        -
        t\partial_\nu\phi_\omega(x)
    \right].
\end{equation}
The limit exists locally uniformly on \(\Gamma_\omega\) by Theorem~\ref{thm:large-data-asymptotic}, and
\begin{equation}\label{eq:L-omega-normal}
    L_\omega^f(x)=\partial_\nu w_\omega^f(x),
    \qquad x\in\Gamma_\omega .
\end{equation}
Thus \(L_\omega^f\) is determined by the Dirichlet-to-Neumann map on the large cylindrical family.

\begin{lemma}\label{lem:DN-equality-L}
Let \(f_1,f_2\in\mathcal F\).  Suppose that, for some \(T_0>0\),
\[
    \DN_{f_1}(t\phi_\omega|_{\bd\Om})
    =
    \DN_{f_2}(t\phi_\omega|_{\bd\Om})
\]
for all \(t\ge T_0\) and all \(\omega\in\Sn\).  Then
\[
    L_\omega^{f_1}=L_\omega^{f_2}
    \quad\text{on }\Gamma_\omega
\]
for every \(\omega\in\Sn\).
\end{lemma}

\begin{proof}
Fix \(\omega\in\Sn\) and \(x\in\Gamma_\omega\).  For all sufficiently large \(t\), the equality of the Dirichlet-to-Neumann data gives
\[
    t^{n-1}
    \left[
        \DN_{f_1}(t\phi_\omega|_{\bd\Om})(x)
        -t\partial_\nu\phi_\omega(x)
    \right]
    =
    t^{n-1}
    \left[
        \DN_{f_2}(t\phi_\omega|_{\bd\Om})(x)
        -t\partial_\nu\phi_\omega(x)
    \right].
\]
Passing to the limit \(t\to\infty\) and using \eqref{eq:L-omega} gives the claim.
\end{proof}

\subsection{Endpoint derivatives and chord integrals}\label{subsec:endpoint-derivatives-chord-integrals}

\begin{definition}[Chordwise X-ray transform]\label{def:xray}
For \(f\in C(\overline\Om)\), define
\begin{equation}\label{eq:xray}
    \Xray_\omega f(y)
    =
    \int_{s_-(y)}^{s_+(y)}
        f(y+s\omega)\,ds,
    \qquad y\in\Omega_\omega.
\end{equation}
\end{definition}

\begin{lemma}\label{lem:endpoint-identity}
Let \(w=w_\omega^f\) be defined by \eqref{eq:w-correction}.  Then, for every \(y\in\Omega_\omega\),
\begin{equation}\label{eq:line-integral-endpoint-derivatives}
    \Xray_\omega f(y)
    =
    \partial_s w(y,s_+(y))
    -
    \partial_s w(y,s_-(y)).
\end{equation}
Moreover, since \(w=0\) on the boundary endpoints of the chords,
\begin{equation}\label{eq:line-integral-normal}
    \Xray_\omega f(y)
    =
    (\omega\cdot\nu(x_+(y)))\,\partial_\nu w(x_+(y))
    -
    (\omega\cdot\nu(x_-(y)))\,\partial_\nu w(x_-(y)).
\end{equation}
\end{lemma}

\begin{proof}
Integrating \(\partial_s^2w=f\) along the chord gives
\[
    \int_{s_-(y)}^{s_+(y)} f(y+s\omega)\,ds
    =
    \partial_s w(y,s_+(y))-
    \partial_s w(y,s_-(y)),
\]
which is \eqref{eq:line-integral-endpoint-derivatives}.  By Lemma~\ref{lem:chord-parametrization}, the endpoints \(x_\pm(y)\) belong to the non-glancing boundary pieces \(\Gamma_\omega^\pm\), and the maps \(y\mapsto x_\pm(y)\) locally parametrize those pieces.  Since \(w(y,s_\pm(y))=0\), the function \(w\) vanishes on these non-glancing boundary pieces.  Thus the tangential derivative of \(w\) along \(\bd\Om\) vanishes at \(x_\pm(y)\).  Therefore \(\nabla w(x_\pm(y))\) is parallel to \(\nu(x_\pm(y))\), and
\[
    \partial_s w(x_\pm(y))
    =
    \omega\cdot\nabla w(x_\pm(y))
    =
    (\omega\cdot\nu(x_\pm(y)))\partial_\nu w(x_\pm(y)).
\]
Substituting these identities into \eqref{eq:line-integral-endpoint-derivatives} proves \eqref{eq:line-integral-normal}.
\end{proof}

\begin{proposition}\label{prop:recover-xray}
Let \(f\in\mathcal F\).  Then the Dirichlet-to-Neumann map on the large cylindrical family determines \(\Xray_\omega f\) for every \(\omega\in\Sn\).  More precisely, for every \(y\in\Omega_\omega\),
\begin{equation}\label{eq:recover-xray}
    \Xray_\omega f(y)
    =
    (\omega\cdot\nu(x_+(y)))\,L_\omega^f(x_+(y))
    -
    (\omega\cdot\nu(x_-(y)))\,L_\omega^f(x_-(y)).
\end{equation}
\end{proposition}

\begin{proof}
By Lemma~\ref{lem:chord-parametrization}, the endpoints \(x_\pm(y)\) belong to the non-glancing sets \(\Gamma_\omega^\pm\).  Hence \eqref{eq:L-omega-normal} applies at both endpoints.  Substituting \(L_\omega^f=\partial_\nu w_\omega^f\) into the endpoint identity \eqref{eq:line-integral-normal} gives \eqref{eq:recover-xray}.
\end{proof}

\begin{corollary}\label{cor:DN-equality-xray}
Let \(f_1,f_2\in\mathcal F\).  If there is \(T_0>0\) such that
\[
    \DN_{f_1}(t\phi_\omega|_{\bd\Om})
    =
    \DN_{f_2}(t\phi_\omega|_{\bd\Om})
\]
for all \(t\ge T_0\) and all \(\omega\in\Sn\), then
\[
    \Xray_\omega f_1(y)=\Xray_\omega f_2(y)
\]
for every \(\omega\in\Sn\) and every \(y\in\Omega_\omega\).
\end{corollary}

\begin{proof}
By Lemma~\ref{lem:DN-equality-L}, the equality of the DN maps implies equality of the limits \(L_\omega^{f_1}=L_\omega^{f_2}\) on \(\Gamma_\omega\).  Proposition~\ref{prop:recover-xray} then gives equality of the corresponding chord integrals.
\end{proof}

\subsection{Glancing points cause no loss of data}\label{subsec:glancing-removal}

The boundary asymptotic is used only on the non-glancing boundary set \(\Gamma_\omega\).  This subsection records explicitly that this loses no information needed for the X-ray step.  For a fixed direction \(\omega\), the only projected lines not represented by interior chords are tangent lines, and their chord length is zero.

\begin{lemma}\label{lem:glancing-projection-density}
Fix \(\omega\in\Sn\).  Then
\begin{equation}\label{eq:projection-closure}
    \overline{\Omega_\omega}=P_{\omega^\perp}\overline\Omega,
\end{equation}
and the projected glancing set is exactly the boundary of the projected domain:
\begin{equation}\label{eq:projected-glancing-boundary}
    P_{\omega^\perp}\Gamma_\omega^{\mathrm{gl}}
    =
    \partial\Omega_\omega .
\end{equation}
Moreover, if \(y_j\in\Omega_\omega\) and \(y_j\to y_0\in\partial\Omega_\omega\), then
\begin{equation}\label{eq:ell-to-zero-at-glancing}
    \ell_\omega(y_j)=s_+(y_j)-s_-(y_j)\longrightarrow 0 .
\end{equation}
Consequently, the positive-length chords \(y\in\Omega_\omega\), whose endpoints are non-glancing by Lemma~\ref{lem:chord-parametrization}, are dense in the set of all projected chords \(\overline{\Omega_\omega}\).  The missing boundary family consists only of zero-length tangent chords.
\end{lemma}

\begin{proof}
Since \(\overline\Omega\) is compact and \(P_{\omega^\perp}\) is continuous, \(P_{\omega^\perp}\overline\Omega\) is compact.  Because \(\Omega\) is dense in \(\overline\Omega\), this compact set is the closure of \(P_{\omega^\perp}\Omega=\Omega_\omega\).  This proves \eqref{eq:projection-closure}.

Let \(x\in\Gamma_\omega^{\mathrm{gl}}\), and set \(y=P_{\omega^\perp}x\).  Then \(y\in\overline{\Omega_\omega}\).  Since \(\omega\cdot\nu(x)=0\), the line \(x+\R\omega\) lies in the tangent hyperplane to \(\partial\Omega\) at \(x\).  Uniform convexity implies strict convexity of \(\overline\Omega\); hence this supporting hyperplane meets \(\overline\Omega\) only at \(x\).  Thus the line \(y+\R\omega\) does not meet \(\Omega\), so \(y\notin\Omega_\omega\).  Therefore \(y\in\partial\Omega_\omega\).

Conversely, let \(y\in\partial\Omega_\omega\).  By \eqref{eq:projection-closure}, there is \(x\in\overline\Omega\) with \(P_{\omega^\perp}x=y\).  Since \(y\notin\Omega_\omega\), this point cannot lie in \(\Omega\), and hence \(x\in\partial\Omega\).  If \(\omega\cdot\nu(x)\ne0\), then the line \(x+\R\omega\) crosses \(\partial\Omega\) transversely at \(x\).  The implicit function theorem, or simply the sign change of a defining function along this transverse line, then gives nearby points of \(\Omega\) with the same projection \(y\), contradicting \(y\notin\Omega_\omega\).  Hence \(x\in\Gamma_\omega^{\mathrm{gl}}\), proving \eqref{eq:projected-glancing-boundary}.

It remains to prove \eqref{eq:ell-to-zero-at-glancing}.  Suppose instead that, after passing to a subsequence, \(\ell_\omega(y_j)\ge\delta>0\).  The endpoints
\[
    x_{\pm,j}=y_j+s_\pm(y_j)\omega
\]
lie in the compact set \(\partial\Omega\), so after passing to a further subsequence they converge to points \(x_\pm\in\partial\Omega\).  Their projections are both \(y_0\), and
\[
    |x_+-x_-|=\lim_{j\to\infty}\ell_\omega(y_j)\ge\delta,
\]
so \(x_+\ne x_-\).  By strict convexity, the open segment joining \(x_-\) to \(x_+\) lies in \(\Omega\).  Every point of this segment projects to \(y_0\), so \(y_0\in\Omega_\omega\), contradicting \(y_0\in\partial\Omega_\omega\).  Therefore \(\ell_\omega(y_j)\to0\).
\end{proof}

\begin{lemma}\label{lem:xray-continuous-extension}
Let \(F\in C(\overline\Omega)\), and fix \(\omega\in\Sn\).  Define
\begin{equation}\label{eq:extended-chord-transform}
    \overline{\Xray}_\omega F(y)
    =
    \begin{cases}
    \displaystyle\int_{s_-(y)}^{s_+(y)}F(y+s\omega)\,ds,& y\in\Omega_\omega,\\[1.2em]
    0,& y\in\omega^\perp\setminus\Omega_\omega.
    \end{cases}
\end{equation}
Then \(\overline{\Xray}_\omega F\) is continuous on \(\omega^\perp\).  In particular, it vanishes on \(\partial\Omega_\omega=P_{\omega^\perp}\Gamma_\omega^{\mathrm{gl}}\).
\end{lemma}

\begin{proof}
Continuity inside \(\Omega_\omega\) follows from the smooth dependence of \(s_\pm(y)\) and the continuity of \(F\).  The function is identically zero on the open set \(\omega^\perp\setminus\overline{\Omega_\omega}\).  It remains to check continuity at \(y_0\in\partial\Omega_\omega\).  If \(y_j\to y_0\) and \(y_j\in\Omega_\omega\), then
\[
    |\overline{\Xray}_\omega F(y_j)|
    \le
    \|F\|_{L^\infty(\Omega)}\ell_\omega(y_j)
    \longrightarrow 0
\]
by Lemma~\ref{lem:glancing-projection-density}.  If \(y_j\notin\Omega_\omega\), the value is already zero.  Hence \(\overline{\Xray}_\omega F\) is continuous at the boundary as well.
\end{proof}

\begin{proposition}\label{prop:no-loss-glancing-xray}
Let \(F\in C(\overline\Omega)\).  Suppose that, for every \(\omega\in\Sn\), the chord integrals \(\Xray_\omega F(y)\) are known for all positive-length chords \(y\in\Omega_\omega\).  Then the full Euclidean X-ray transform of the zero extension \(\widetilde F\) is known for all \(y\in\omega^\perp\): it is \(\Xray_\omega F(y)\) on \(\Omega_\omega\) and is zero on \(\omega^\perp\setminus\Omega_\omega\).  In particular, if \(\Xray_\omega F=0\) on \(\Omega_\omega\) for every \(\omega\), then \(F=0\) in \(\Omega\).
\end{proposition}

\begin{proof}
For \(y\in\Omega_\omega\), the line \(y+\R\omega\) intersects \(\Omega\) in the chord \((s_-(y),s_+(y))\), so the Euclidean X-ray transform of the zero extension \(\widetilde F\) equals \(\Xray_\omega F(y)\).  If \(y\notin\Omega_\omega\), the line does not meet the open set \(\Omega\); at boundary points of \(\Omega_\omega\) it is a tangent line and contributes zero to the one-dimensional integral.  Thus the missing glancing values are forced to be zero, consistently with the continuous extension in Lemma~\ref{lem:xray-continuous-extension}.  The final assertion follows from Lemma~\ref{lem:xray-injective}.
\end{proof}

\subsection{Injectivity of the X-ray transform}\label{subsec:xray-injectivity}

\begin{lemma}\label{lem:xray-injective}
Let \(F\in C(\overline\Om)\).  If
\[
    \Xray_\omega F(y)=0
\]
for every \(\omega\in\Sn\) and every \(y\in \Omega_\omega\), then \(F=0\) in \(\Om\).
\end{lemma}

\begin{proof}
Extend \(F\) by zero to a compactly supported function \(\widetilde F\in L^1(\R^n)\).  For fixed \(\omega\), define the Euclidean X-ray transform
\[
    R_\omega \widetilde F(y)
    =
    \int_\R \widetilde F(y+s\omega)\,ds,
    \qquad y\in\omega^\perp .
\]
If \(y\in\Omega_\omega\), this is precisely \(\Xray_\omega F(y)\), and if \(y\notin\Omega_\omega\), the line \(y+\R\omega\) does not meet \(\Omega\), so \(R_\omega\widetilde F(y)=0\).  Hence \(R_\omega\widetilde F=0\) for every \(\omega\).

Taking the Fourier transform in the \(y\)-variable on \(\omega^\perp\), one obtains the Fourier slice identity; this is the standard Fourier-slice proof of injectivity for the X-ray/Radon transform; see \cite{Natterer2001}:
\[
    \widehat{R_\omega\widetilde F}(\xi)
    =
    \widehat{\widetilde F}(\xi),
    \qquad \xi\in\omega^\perp .
\]
Indeed,
\[
\begin{aligned}
    \widehat{R_\omega\widetilde F}(\xi)
    &=\int_{\omega^\perp} e^{-iy\cdot\xi}
        \int_\R \widetilde F(y+s\omega)\,ds\,dy  \\
    &=\int_{\R^n} e^{-ix\cdot\xi}\widetilde F(x)\,dx
     =\widehat{\widetilde F}(\xi),
\end{aligned}
\]
because \(\xi\cdot\omega=0\).  Given any \(\xi\ne0\), choose \(\omega\in\Sn\) with \(\xi\in\omega^\perp\).  Since \(R_\omega\widetilde F=0\), the identity gives \(\widehat{\widetilde F}(\xi)=0\).  Since \(\widetilde F\in L^1(\R^n)\), its Fourier transform is continuous; hence the same conclusion holds at \(\xi=0\).  Thus \(\widehat{\widetilde F}\equiv0\), and hence \(\widetilde F=0\), so \(F=0\) in \(\Omega\).
\end{proof}

\section{Proof of uniqueness}\label{sec:uniqueness-proofs}

\begin{theorem}\label{thm:large-cylindrical-uniqueness}
Let \(\Om\subset\R^n\), \(n\ge2\), be bounded, smooth, and uniformly convex.  Let
\[
    f_1,f_2\in C^\infty(\overline\Om),
    \qquad
    0<c\le f_j\le C<\infty.
\]
Assume that there is \(T_0>0\) such that
\begin{equation}\label{eq:DN-equality-large-family}
    \DN_{f_1}(g_{t,\omega})=
    \DN_{f_2}(g_{t,\omega})
\end{equation}
for all \(t\ge T_0\) and all \(\omega\in\Sn\).  Then
\[
    f_1=f_2
    \quad\text{in }\Om.
\]
\end{theorem}

\begin{proof}[Proof of Theorem~\ref{thm:large-cylindrical-uniqueness}]
By Corollary~\ref{cor:DN-equality-xray}, the equality of the Dirichlet-to-Neumann maps on the large cylindrical family implies
\[
    \Xray_\omega f_1(y)=\Xray_\omega f_2(y)
\]
for every \(\omega\in\Sn\) and every \(y\in\Omega_\omega\).  Equivalently,
\[
    \Xray_\omega(f_1-f_2)(y)=0
\]
for every positive-length chord of \(\Omega\).  Proposition~\ref{prop:no-loss-glancing-xray} shows that the missing glancing values cause no loss for the Euclidean X-ray transform of the zero extension.  Lemma~\ref{lem:xray-injective}, applied to \(F=f_1-f_2\), gives \(f_1=f_2\) in \(\Omega\).
\end{proof}

\begin{proof}[Proof of Theorem~\ref{thm:main}]
For each fixed \(\omega\), the trace \(g_{t,\omega}=t\phi_\omega|_{\partial\Omega}\) is a smooth boundary value for every \(t>0\).  The assumed equality of the full Dirichlet-to-Neumann maps therefore gives
\[
    \DN_{f_1}(g_{t,\omega})=\DN_{f_2}(g_{t,\omega})
\]
for all \(t>0\) and all \(\omega\in S^{n-1}\).  In particular, the hypothesis of Theorem~\ref{thm:large-cylindrical-uniqueness} holds for any \(T_0>0\), and that theorem gives \(f_1=f_2\) in \(\Omega\).
\end{proof}

\section*{Acknowledgements}
C.I.C. was supported by NSTC grant 113-2115-M-A49-018-MY3. T.G. was supported by grant number NBHM(R.P)/R\&D II/9464.

\bibliographystyle{abbrv}
\bibliography{monge_ampere_refs_v5}

\end{document}